\documentclass[10pt]{amsart}
\usepackage[utf8]{inputenc}

\usepackage{amssymb,amsthm,amsmath}
\usepackage{mathrsfs}
\usepackage{enumerate}
\usepackage{graphicx,xcolor}
\usepackage{setspace}
\usepackage[hidelinks]{hyperref}

\newcommand{\dd}{\mathrm{d}}
\newcommand{\E}{\mathbb{E}}

\newcommand{\1}{\textbf{1}}
\newcommand{\R}{\mathbb{R}}
\newcommand{\C}{\mathbb{C}} 

\newcommand{\e}{\varepsilon}
\newcommand{\p}[1]{\mathbb{P}\left( #1 \right)}
\newcommand{\scal}[2]{\left\langle #1, #2 \right\rangle}

\usepackage{xpatch}
\makeatletter
\def\thm@space@setup{%
  \thm@preskip=12pt plus 0pt minus 0pt
  \thm@postskip=0pt plus 0pt minus 0pt
}
\xpatchcmd{\proof}{6\p@\@plus6\p@\relax}{\z@skip}{}{}
\makeatother

\newtheorem{theorem}{Theorem}
\newtheorem{lemma}[theorem]{Lemma}
\newtheorem{corollary}[theorem]{Corollary}

\theoremstyle{remark}
\newtheorem{remark}[theorem]{Remark}

\theoremstyle{definition}

\title{Complex Hanner's inequality for many functions}

\author{Jonathan Jenkins}

\address{(J.J.) Carnegie Mellon University; Pittsburgh, PA 15213, USA and Pennsylvania State University; State College, PA 16801, USA.}

\email{jtj5311@psu.edu}

\author{Tomasz Tkocz}

\address{(T.T) Carnegie Mellon University; Pittsburgh, PA 15213, USA.}

\email{ttkocz@andrew.cmu.edu}

\thanks{TT's research supported in part by NSF grant DMS-1955175.}

\date{\today}

\begin{document}

\begin{abstract} 
We establish Hanner's inequality for arbitrarily many functions in the setting where the Rademacher distribution is replaced with higher dimensional random vectors uniform on Euclidean spheres.
\end{abstract}

\maketitle

\bigskip

\begin{footnotesize}
\noindent {\em 2010 Mathematics Subject Classification.} Primary 26D15; Secondary 52A41.

\noindent {\em Key words. } $L_p$ spaces, Hanner's inequality, uniform convexity, uniform spherically symmetric random vectors.
\end{footnotesize}

\bigskip

\section{Introduction and the main result}

Classical Hanner's inequality from \cite{Ha} states that for \emph{two} functions $f_1, f_2$ in $L_p$, we have
\begin{equation}\label{eq:Hanner}
\|f_1 + f_2\|_p^p + \|f_1 - f_2\|_p^p \leq \big|\|f_1\|_p + \|f_2\|_p\big|^p + \big|\|f_1\|_p - \|f_2\|_p\big|^p,
\end{equation}
when $p \geq 2$ and the reverse inequality holds when $1 \leq p \leq 2$. Throughout this paper,  without loss of generality we shall assume that the $L_p$ space is of \emph{real} functions on $[0,1]$ equipped with Lebesgue measure, with the underlying norm $\|f\|_p = (\int_0^1 |f(t)|^p \dd t)^{1/p}$. This inequality was discovered in relation to uniform convexity introduced by Clarkson in \cite{C} and gave an optimal result in this direction for $L_p$ spaces. Hanner's inequality has been influential and its various generalisations and sharpenings have been extensively studied, see, e.g. \cite{BCL, CFIL, CFL, IM, Tru}. 

One enthralling question concerns extensions of Hanner's inequality to \emph{many} functions. As exemplified by Schechtman in \cite{Sch}, one possible version is intimately connected to $p$-concavity constants of Banach lattices generated by unconditional basic sequences in $L_p$ spaces. In elementary terms, his main result says that for every $p \geq 3$, integer $n \geq 1$  and functions $f_1, \dots, f_n$ in $L_p$, we have
\begin{equation}\label{eq:Hanner-many}
\E\left\|\sum_{k=1}^n \e_k f_k\right\|_p^p \leq \E\left|\sum_{k=1}^n \e_k \|f_k\|_p\right|^p,
\end{equation}
where $\e_1, \e_2, \dots$ are independent identically distributed (i.i.d.) Rademacher random variables; that is $\p{\e_k = \pm 1} = \frac{1}{2}$. It is natural to conjecture that this inequality continues to hold for all $p \geq 2$, whilst for all $1 \leq p \leq 2$ the reversal holds (note that for $p=2$, we have equality by the parallelogram identity). To the best of our knowledge, both cases remain open. Plainly, $n=2$ reduces to Hanner's inequality \eqref{eq:Hanner}. 

For \emph{complex} Banach spaces, whenever the complex structure plays an important role, the Rademacher two-point distribution is replaced with the Steinhaus distribution (the uniform distribution on the complex unit circle $\{z \in \C, \ |z| = 1\}$), see, e.g. \cite{BD, DGT, DS, Ver1, Ver2}. Motivated by this, the main goal of this note is to show a complex analogue of \eqref{eq:Hanner-many}, where the Rademacher distribution in \eqref{eq:Hanner-many} is replaced by the Steinhaus distribution. In fact, we obtain the following multidimensional result, which when specialised to $d=2$ gives the complex analogue. Here and throughout, $\scal{\cdot}{\cdot}$ denotes the standard inner product, $|\cdot|$ the standard Euclidean norm on $\R^d$ and $S^{d-1} = \{x \in \R^n, \ |x| = 1\}$ is its unit sphere. 

\begin{theorem}\label{thm:main}
Let $d \geq 2$ and let $\xi_1, \xi_2, \dots$ be i.i.d. random vectors uniform on the unit Euclidean sphere $S^{d-1}$ in $\R^d$. Let $p \geq 2$. For every $n \geq 1$ and functions $f_1, \dots, f_n$ in $L_p$, we have
\begin{equation}\label{eq:main}
\E\left\|\sum_{k=1}^n \xi_k f_k\right\|_p^p\leq \E\left|\sum_{k=1}^n \xi_k \|f_k\|_p\right|^p.
\end{equation}
Assuming that $d \geq 3$, the reverse inequality holds when $1 \leq p \leq 2$.
\end{theorem}

When $1 \leq p \leq 2$, we conjecture that the reverse inequality continues to hold for $d=2$.

We present the proof of Theorem \ref{thm:main} in the next section, after which we conclude the paper with several remarks, most notably regarding connections to basic unconditional sequences in complex $L_p$ spaces.



\section{Proofs}

\subsection{Overview}
We follow Schechtman's approach from \cite{Sch}, which can be summarised as the following steps: 1) Restating \eqref{eq:Hanner-many} equivalently as the concavity of a function on $\R_+^n$, 2) Analysis of the Hessian, 3) Point-wise bounds (resulting from the convexity of $|\cdot|^{p-2}$). Some additional ideas are needed to adapt these steps to multivariate distributions but it ought to be highlighted that it is mainly rotational symmetry and homogeneity that underpin and make the adaptation and reduction to one dimension possible. We detail these steps and additionally recall the crucial notions of unimodality and rotational invariance in the next $4$ subsections. The final argument for dimensions $d \geq 3$ differs significantly from dimension $d=2$: the former relying on the unimodality of marginal distributions which is absent when $d=2$, in which case we leverage a complex analytic point of view. This occupies the last subsection.

\subsection{Equivalent forms}

We begin by remarking that thanks to a standard argument relying on Jensen's inequality, Theorem \ref{thm:main} becomes equivalent to the concavity/convexity of a certain $1$-homogeneous function.

\begin{lemma}\label{lm:equiv}
Under the assumptions of Theorem \ref{thm:main}, inequality \eqref{eq:Hanner-many} holds for some $n \geq 1$ and every $f_1, \dots, f_n \in L_p$ if and only if the following function
\begin{equation}\label{eq:phi}
\phi_n(x_1, \dots, x_n) = \E\left|\sum_{k=1}^n x_k^{1/p}\xi_k\right|^p
\end{equation}
is concave on $\R_+^n$, and the reverse inequality holds if and only if this function is convex.
\end{lemma}
\begin{proof}
    We prove only the concave case since the convex case uses the same logic. Note that inequality \eqref{eq:Hanner-many} holds if and only if
    \[
        \int_0^1 \phi_n(|f_1|^p,\dots,|f_n|^p) \: dt \leq\phi_n(\lVert f_1 \rVert_p^p, \dots, \lVert f_n \rVert_p^p) 
    \]
    for all $f_1, \dots, f_n \in L_p$. This comes from merely translating inequality \eqref{eq:Hanner-many} into our new notation using $\phi_n$. Suppose this inequality holds for all $f_1, \dots, f_n$. Then for each $x,y \in (0,+\infty)$, if we define $f_i = x_i^{1/p}$ with probability $\lambda$ and $f_i = y_i^{1/p}$ with probability $1-\lambda$, then the inequality says exactly that
    \[
        \lambda \phi_n(x) + (1-\lambda) \phi_n(y) \leq \phi (\lambda x + (1-\lambda) y)
    \]
    Thus the inequality implies that $\phi_n$ is concave. To see the reverse direction, assume that $\phi_n$ is concave. Then our specified inequality follows exactly from Jensen's inequality. This concludes the proof of the lemma.
\end{proof}
Note that the expectation $\E\left|\sum_{k=1}^n a_k\xi_k\right|^p$ exists for all $n \geq 1$ and $a_1, \dots, a_n \in \R$ as long as $p > -(d-1)$. Moreover, as a result of rotational invariance, for every $p > -1$, $n \geq 1$ and $a_1, \dots, a_n \in \R$, we have
\begin{equation}\label{eq:prop}
\E\left|\sum_{k=1}^n a_k\xi_k\right|^p = \beta_{p,d}\E\left|\sum_{k=1}^n a_k\scal{\xi_k}{e_1}\right|^p
\end{equation}
with the constant
\[
\beta_{p,d} = \frac{1}{\E|\scal{\xi_1}{e_1}|^p} = \frac{\sqrt{\pi}\Gamma\left(\frac{d+p}{2}\right)}{\Gamma\left(\frac{d}{2}\right)\Gamma\left(\frac{p+1}{2}\right)}.
\]
This can perhaps be traced back to work of K\"onig and Kwapie\'n (see Lemma 8 in \cite{KK}) and follows from the identity
\begin{equation}\label{eq:embedding}
|x|^p = \beta_{p,d}\E|\scal{x}{\eta}|^p,
\end{equation}
which holds for every fixed vector $x$ in $\R^d$ and the expectation is taken with respect to $\eta$, a random vector uniform on the unit sphere.
We shall denote 
\[
\theta_k = \scal{\xi_k}{e_1}
\]
which are i.i.d. \emph{symmetric} random variables with density 
\begin{equation}\label{eq:density}
\frac{\Gamma\left(\frac{d}{2}\right)}{\sqrt{\pi}\Gamma\left(\frac{d-1}{2}\right)}(1-x^2)^{\frac{d-3}{2}}\1_{|x|\leq 1}.
\end{equation}
We recall that a random variable $X$ is called symmetric if $-X$ has the same distribution as $X$ (equivalently, $X$ is a mixture of scaled Rademacher distributions). In view of \eqref{eq:phi}, identity \eqref{eq:prop} leads to
\begin{equation}\label{eq:phi-proj}
\phi_n(x_1, \dots, x_n) = \beta_{p,d}\E\left|\sum_{k=1}^n x_k^{1/p}\theta_k\right|^p.
\end{equation}
Thus as long as $d \geq 2$ and $p > 1$, functions $\phi_n$ are $C^2$ on $(0,+\infty)^n$ and the convexity/concavity of $\phi_n$ is equivalent to its Hessian matrix being positive/negative semi-definite.

\subsection{Hessian}
The following lemma provides a handy expression for quadratic forms given by Hessian matrices of functions of the form \eqref{eq:phi-proj}.

\begin{lemma}\label{lm:Hess}
Let $Y_1, \dots, Y_n$ be i.i.d. symmetric random variables such that $\E|Y_1|^{-1+\delta}$ is finite for some $\delta > 0$. Let $p \geq 1+\delta$. Define $f\colon (0,+\infty)^n \to (0,+\infty)$,
\[
f(x_1, \dots, x_n) = \E\left|\sum_{k=1}^n x_k^{1/p}Y_k\right|^p.
\]
Then $f$ is $C^2$ and for every vectors $a \in \R^n$ and $x \in (0,+\infty)^n$, we have
\[
\sum_{k,l=1}^n a_ka_l\frac{\partial^2 f}{\partial x_k\partial x_l}(x) = -\frac{p-1}{p}\sum_{1 \leq k < l \leq n} \left(\frac{a_k}{x_k}-\frac{a_l}{x_l}\right)^2(x_kx_l)^{1/p}\E\left[\left|\sum_{j=1}^nx_j^{1/p}Y_j\right|^{p-2}Y_kY_l\right].
\]
\end{lemma}
\begin{proof}
    This lemma is implicit in Schechtman's work \cite{Sch}, in that it is proved for Rademacher random variables. For the sake of clarity we reprove it here. The main idea is that if we define $\gamma_{i,j} = \E \left| \sum_{k=1}^n x_k^{1/p} Y_k\right|^{p-2} Y_i Y_j $, then we can compute that for $i\neq j$, we have that
    \[
    \frac{\partial^2f}{\partial x_i \partial x_j}(x_1,\dots,x_n) =\frac{p-1}{p} \gamma_{i,j} x_i^{\frac{1-p}{p}} x_j^{\frac{1-p}{p}} 
    \]
    \[
    \frac{\partial^2f}{\partial x_i^2}(x_1,\dots,x_n) =\frac{p-1}{p} \gamma_{i,i} \left( x_i^{\frac{1-p}{p}} \right)^2- \frac{p-1}{p}\sum_{k=1}^n x_i^{\frac{1-2p}{p}} x_k^{1/p}\gamma_{i,k}
    \]
    Accordingly for any vector $a \in \R^n$ and $x \in (0,+\infty)^n$, we have that
    \begin{align*}
        \sum_{k,l=1}^n a_k a_l \frac{\partial^2 f}{\partial x_k \partial x_l} &= -\frac{p-1}{p} \left(\sum_{k, l = 1}^n \gamma_{k,l} x_k^{1/p}x_l^{1/p} (a_k^2 x_k^{-2} -a_k a_l x_k^{-1} x_l^{-1} ) \right)\\
        &= -\frac{p-1}{p} \sum_{1\leq k < l \leq n} \gamma_{k,l} (x_k x_l)^{1/p} (a_k^2 x_k^{-2} -2 a_k a_l x_k^{-1} x_l^{-1} + a_l^2 x_l^{-2}) \\
        &= -\frac{p-1}{p}\sum_{1 \leq k < l \leq n} \left(\frac{a_k}{x_k}-\frac{a_l}{x_l}\right)^2(x_kx_l)^{1/p}\E\left[\left|\sum_{j=1}^nx_j^{1/p}Y_j\right|^{p-2}Y_kY_l\right]
    \end{align*}
    This concludes the proof of the lemma.
\end{proof}

\subsection{Two-point inequalities}
In view of the expression from Lemma \ref{lm:Hess}, it is natural to investigate the sign of the expectations $\E\left[\left|\sum x_j^{1/p}Y_j\right|^{p-2}Y_kY_l\right]$. The final ingredients of the whole argument are some rather general two-point inequalities, resulting from the monotonicity of certain functions. We say that a random vector $X = (X_1, \dots, X_d)$ is unconditional if $(\e_1|X_1|, \dots, \e_d|X_d|)$ has the same distribution as $X$, where $\e_1, \dots, \e_d$ are i.i.d. Rademacher random variables, independent of $X$.

\begin{lemma}\label{lm:2pt-gen}
Let $X$ and $Y$ be independent unconditional random vectors in $\R^d$ and let $h\colon \R_+ \to \R_+$ be a nondecreasing function. Then (provided the expectation exists)
\[
\E\left[h(|X+Y|)\scal{X}{Y}\right] \geq 0.
\]
\end{lemma}
\begin{proof}
Since $\scal{X}{Y} = \sum_{k=1}^d X_kY_k$, it suffices to prove that for each $k$,
\[
\E\left[h(|X+Y|)X_kY_k\right] \geq 0.
\]
By independence and unconditionality, we can write, 
\[
\E [h(|X+Y|)X_kY_k] = \E\left[\frac{h\left(\sqrt{(|X_k|+|Y_k|)^2+R}\right)-h\left(\sqrt{(|X_k|-|Y_k|)^2+R}\right)}{2}|X_k||Y_k|\right]
\]
with $R = \sum_{j\neq k} (X_j+Y_j)^2$.
The lemma now follows since $|u|+|v| \geq \big||u|-|v|\big|$ for all real numbers $u, v$.
\end{proof}

To use this lemma, we need two elementary facts about the monotonicity of relevant functions arising from the uniform distribution on Euclidean spheres. In dimension 2, we use a complex analytic argument.

\begin{lemma}\label{lm:monot-steinhaus}
Let $q > 0$. Let $\xi$ be a $\C$-valued random variable uniform on the unit circle $\{z \in \C, \ |z| = 1\}$. The function
\[
g_q(z) = \E|z+\xi|^q, \qquad z \in \C,
\]
is radial and increasing on $[0, +\infty)$.
\end{lemma}

\begin{proof}
By the rotational invariance of $\xi$, $g_q(e^{it}z) = g_q(z)$ for every $t \in \R$ and $z \in \C$, thus $g_q(z)$ depends only on $|z|$. We also note that $g_q$ is continuous as follows for instance from Lebesgue's dominated convergence theorem. This allows to consider separately two cases.

For $0 < x < 1$, we have (using the principal branch of the logarithm and the binomial series)
\begin{align*}
g_q(x) = \E |x+\xi|^q &= \E |x \xi + 1|^q =\frac{1}{2\pi}\int_0^{2\pi} |1+xe^{it}|^q \dd t \\&= \frac{1}{2\pi}\int_0^{2\pi} (1+xe^{it})^{q/2}(1+xe^{-it})^{q/2} \dd t \\
&= \frac{1}{2\pi}\int_0^{2\pi}\sum_{m,n=0}^\infty \binom{q/2}{m}\binom{q/2}{n}x^{m+n}e^{it(m-n)} \dd t \\
&= \sum_{n=0}^\infty \binom{q/2}{n}^2x^{2n},
\end{align*}
so $g_q(x)$ is clearly increasing on $(0,1)$ (this argument also appears e.g. in \cite{Ko}, see (27) therein).  

For $x > 1$, we write
\[
g_q(x) = \frac{1}{2\pi}\int_0^{2\pi} |x+e^{it}|^q \dd t = \frac{1}{2\pi}\int_0^{2\pi} (x^2+2x\cos t + 1)^{q/2} \dd t.
\]
Taking the (real) derivative yields
\[
g_q'(x) = \frac{q}{2\pi}\int_0^{2\pi} (x^2+2x\cos t + 1)^{q/2-1}(x+\cos t) \dd t > 0,
\]
since plainly $x + \cos t > 1 + \cos t > 0$. 
\end{proof}

\begin{remark}\label{rem:Steinh-neg}
For $-1 < q < 0$, this proof shows that $g_q'(x) > 0$ for $0 < x < 1$, whereas $g_q'(x) < 0$ for $x > 1$.
\end{remark}

In contrast to dimension $2$, in higher dimensions $d \geq 3$, the marginal distributions of vectors uniform on spheres are unimodal (recall \eqref{eq:density}), which will allow to also address the range $1 < p < 2$, with the aid of the following simple lemma.

\begin{lemma}\label{lm:monot-unimod}
Let $q>-1$. Let $U$ be a random variable uniform on $[-1,1]$ and set
\[
h_q(x) = \E|x+U|^q, \qquad x \in \R.
\]
This is an even function, decreasing on $[0,+\infty)$ when $-1 < q < 0$ and increasing on $[0,+\infty)$ when $q > 0$.
\end{lemma}

\begin{proof}
The evenness follows from the symmetry. We have,
\[
h_q'(x) = \frac{1}{2}\frac{\dd}{\dd x}\int_{-1}^1 |x+u|^q \dd u = \frac12\Big(|x+1|^q-|x-1|^{q}\Big)
\]
For $x > 0$, plainly $|x+1| = x+1 > |x-1|$, which gives that $h_q'(x) < 0$ when $-1 < q < 0$ and $h_q'(x) > 0$ when $q >0$.
\end{proof}

\subsection{Unimodality and rotational invariance}
A random variable $X$ is called unimodal if its distribution is a mixture of scaled unimodal distributions, that is $X$ has the same distribution as $RU$ for some nonnegative random variable $R$ and an independent uniform $[-1,1]$ random variable $U$. Equivalently, $X$ has a density which is even and nonincreasing on $[0,+\infty)$. Crucially, sums of independent unimodal random variables are unimodal. We refer for instance to \cite{LO}.

A random vector $X$ in $\R^d$ is rotationally invariant if its distribution is invariant with respect to the orthogonal transformations, equivalently if it is a mixture of random vectors uniform on centred Euclidean spheres, that is $X$ has the same distribution as $|X|\eta$, where $\eta$ is uniform on the unit sphere, independent of $X$.

\subsection{Proof of Theorem \ref{thm:main}}
Fix $p > 1$, $n \geq 1$ and let $\theta_1, \dots, \theta_n$ be i.i.d. random variables with density \eqref{eq:density}. After combining Lemma \ref{lm:equiv}, identity \eqref{eq:phi-proj} and Lemma \ref{lm:Hess}, it suffices to show that for $x_1, \dots, x_n > 0$, the expectations
\begin{equation}\label{eq:Ekl}
E_{k,l} = \E\left[\left|\sum_{j=1}^nx_j^{1/p}\theta_j\right|^{p-2}\theta_k\theta_l\right], \qquad k < l,
\end{equation}
are nonnegative when $p \geq 2$ and nonpositive when $p < 2$, assuming additionally that $d \geq 3$. Denote $q = p-2$ and fix $k < l$. 

\emph{Case $d \geq 3$.}
As sums of independent unimodal random variables are unimodal, the sum $\sum_{j \neq k,l} x_j^{1/p}\theta_j$ is of the form $RU$ for some nonnegative random variable $R$ and a random variable $U$ uniform on $[-1,1]$, independent of $R$. It thus follows from Lemma \ref{lm:monot-unimod} that the function
\[
h_q(x) = \E\left|x + \sum_{j \neq k,l} x_j^{1/p}\theta_j\right|^q, \qquad x \in \R,
\]
is decreasing when $-1 < q < 0$ and increasing when $q \geq 0$ (writing $h_q(x) = \E_R\E_U|x+RU|^q$). Thanks to independence,
\[
E_{k,l} = \E_{\theta_k, \theta_l} \left[h_q\left(x_k^{1/p}\theta_k + x_l^{1/p}\theta_l\right)\theta_k\theta_l\right]
\]
and the claim follows from Lemma \ref{lm:2pt-gen}, applied to ($1$-dimensional) variables $X = x_k^{1/p}\theta_k$ and $Y = x_l^{1/p}\theta_l$ which are symmetric. Note that we are able to apply Lemma \ref{lm:2pt-gen} here because \(h_q\) is even hence \(h_q\left(x_k^{1/p}\theta_k + x_l^{1/p}\theta_l\right) = h_q\left( \left|x_k^{1/p}\theta_k + x_l^{1/p}\theta_l\right|\right)\).

\emph{Case $d = 2$.}
Let $\xi_1, \dots, \xi_n$ and $\eta$ be i.i.d. random vectors which are all uniformly distributed on the unit circle. Using identity \eqref{eq:embedding}, we have
\begin{align*}
\E\left[\left|\sum_{j=1}^nx_j^{1/p}\xi_j\right|^{q}\scal{\xi_k}{\xi_l}\right] &= 
 \beta_{q,d}\E\left[\left|\sum_{j=1}^nx_j^{1/p}\scal{\xi_j}{\eta}\right|^{q}\scal{\xi_k}{\xi_l}\right].
\end{align*}
Now, $\eta$ has the same distribution as $Qe_1$ with $Q$ being a random orthogonal $2 \times~2$ matrix (chosen uniformly, i.e. according to the Haar measure). Noting that $\scal{\xi_k}{\xi_l} = \scal{Q^\top\xi_k}{Q^\top\xi_l}$ and that the $n$-tuple $(\xi_1, \dots, \xi_n)$ has the same distribution as $(Q^\top\xi_1, \dots, Q^\top\xi_n)$, we obtain
\[
\E\left[\left|\sum_{j=1}^nx_j^{1/p}\scal{\xi_j}{\eta}\right|^{q}\scal{\xi_k}{\xi_l}\right] = \E\left[\left|\sum_{j=1}^nx_j^{1/p}\scal{\xi_j}{e_1}\right|^{q}\scal{\xi_k}{\xi_l}\right].
\]
Finally since $\scal{\xi_k}{\xi_l} = \scal{\xi_k}{e_1}\scal{\xi_l}{e_1}+\scal{\xi_k}{e_2}\scal{\xi_l}{e_2}$, using linearity and the fact that the expectation
$\E\left[\left|\sum_{j=1}^nx_j^{1/p}\scal{\xi_j}{e_1}\right|^{q}\scal{\xi_k}{e_2}\scal{\xi_l}{e_2}\right]$
vanish (by unconditionality), we arrive at
\[
E_{k,l} = \beta_{q,2}^{-1}\E\left[\left|\sum_{j=1}^nx_j^{1/p}\xi_j\right|^{q}\scal{\xi_k}{\xi_l}\right].
\] 
To show that the right hand side is nonnegative, it remains to combine Lemmas \ref{lm:2pt-gen} and \ref{lm:monot-steinhaus}. Indeed, the random vector $Y = \sum_{j \neq k,l}^nx_j^{1/p}\xi_j$ is rotationally invariant, thus the function
\[
h_q(z) = \E|z+Y|^q,
\]
is radial and by Lemma \ref{lm:monot-steinhaus} it is nondecreasing in $|z|$. Thanks to independence,
\[
\E\left[\left|\sum_{j=1}^nx_j^{1/p}\xi_j\right|^{q}\scal{\xi_k}{\xi_l}\right] = \E_{\xi_k, \xi_l}\left[h_q\left(\left|x_k^{1/p}\xi_k + x_l^{1/p}\xi_l\right|\right)\scal{\xi_k}{\xi_l}\right]
\]
and Lemma \ref{lm:2pt-gen} finishes the proof.\hfill$\square$

\section{Concluding remarks}

\subsection{Basic sequences}
Recall that a sequence $(x_n)_{n=1}^\infty$ in a Banach space $(B, \|\cdot\|)$ over $\mathbb{K} = \R \text{ or } \C$ is called \emph{basic} if it is a Schauder basis of the closure of its (linear) span. Additionally, it is called \emph{$\mathbb{K}$-$1$-unconditional}, if for every sequence $(a_n)_{n=1}^\infty$ in $\mathbb{K}$ such that $\sum a_nx_n$ converges and every sequence $(\xi_n)_{n=1}^\infty$ of numbers in $\mathbb{K}$ with modulus $1$, we have
\[
\left\|\sum \xi_na_nx_n \right\| = \left\|\sum a_nx_n \right\|.
\]
This definition should be contrasted to the standard definition of $1$-unconditional sequences which only allows the scalar sequence to take values in $\{-1,1\}$ regardless of the underlying field. When $\mathbb{K} = \R$, this notion reduces to standard $1$-unconditionality, and when $\mathbb{K} = \C$, we have that $\C$-$1$-unconditional series are properly $1$-unconditional (with the converse not holding in general). As such, in general $\mathbb{K}$-$1$-unconditionality is a stronger property than $1$-unconditionality. For an example, a sequence of i.i.d. Rademacher/Steinhaus random variables in $L_p$ is $\R/\C$-1-unconditional.

Let $p \geq 1$. Expressing a relevant notion from lattice theory explicitly, we say that a $\mathbb{K}$-$1$-unconditional basic sequence \((x_n)\) in a normed space \((X, \lVert \cdot \rVert )\) is $p$-concave with constant $C$ if the following Hanner-type inequality holds:
\[
\left(\int_0^1 \left\|\sum_{n=1}^N f_n(t) x_n\right\|^p \dd t\right)^{1/p} \leq C\left\|\sum_{n=1}^N \|f_n\|_px_n\right\|
\]
for every $N \geq 1$ and every functions $f_1, \dots, f_N$ in $L_p$ (see, e.g. \cite{LT} and \cite{Sch}). For instance, Schechtman's result \eqref{eq:Hanner-many} gives that the $p$-concavity constant of the Rademacher sequence in $B = L_p(\Omega, \mathbb{P})$ equals $1$, when $p \geq 3$. 

For $B = L_p$ with $\mathbb{K} = \R$ and $p \geq 2$, Schechtman in \cite{Sch} established several equivalent conditions for the Rademacher sequence having $p$-concavity constant $1$ (see Lemma 1 therein). This and all arguments from his paper transfer verbatim to the complex setting and combined with Theorem \ref{thm:main} yield the following result.

\begin{corollary}\label{cor:lattice}
Let $2 \leq p \leq \infty$. For $L_p$ spaces over $\C$, we have

(i) The $p$-concavity constant of the Steinhaus sequence in $L_p$ is $1$.

(ii) The $p$-concavity constant of every $\C$-$1$-unconditional basic sequence in $L_p$ is $1$.

(iii) Every $\C$-$1$-unconditional basic sequence $(x_n)_{n=1}^\infty$ in $L_p$ normalised such that $\|x_n\|_p = 1$ for every $n \geq 1$ satisfies
\[
\left\|\sum_{n=1}^N a_n x_n\right\|_p \leq \left\|\sum_{n=1}^N a_n\xi_n\right\|_p
\]
for all $N \geq 1$ and complex numbers $a_1, \dots, a_N$, where $(\xi_n)_{n=1}^\infty$ are i.i.d. Steinhaus random variables.
\end{corollary}

\subsection{Missing range}
In the setting of Theorem \ref{thm:main}, in the Rademacher case ($d=1$), it is an open problem (to the best of our knowledge) to prove Theorem 1's \eqref{eq:main} for $2 < p < 3$ as well as the opposite inequality for $1 < p < 2$ (\cite{Tak} has an error: Jensen's inequality is incorrectly used in the proof of Theorem 1). In the Steinhaus case ($d=2$), only the range $1 < p < 2$ is left open. It is no longer true that all expectations $E_{k,l}$ from \eqref{eq:Ekl} have the desired sign, and thus, if following the Hessian approach, we would need a coarser grouping of the terms in the sum from Lemma \ref{lm:Hess}, which has been elusive.

\end{document}